\documentclass[10pt,a4paper,times]{article}
\usepackage[latin1,applemac]{inputenc}
\usepackage{amsmath, amssymb, amsfonts, amsthm, amscd, epsfig, multicol, marvosym, float,mathrsfs,enumerate,appendix,float,color}
\usepackage[left=2cm,top=3cm,bottom=3cm,right=2cm]{geometry}
\usepackage[mathscr]{euscript}
\theoremstyle{definition}
\newcounter{dummy} \numberwithin{dummy}{section}
\newtheorem{teorema}[dummy]{Theorem}

\newtheorem{corollario}[dummy]{Corollary}
\newtheorem{lemma}[dummy]{Lemma}

\newtheorem{osservazione}[dummy]{Remark}

\numberwithin{equation}{section}
\begin{document}
\title{New hypergeometric formul\ae  \   \ to $\boldsymbol{\pi}$  arising from M. Roberts hyperelliptic reductions}
\author{Giovanni Mingari Scarpello \footnote{giovannimingari@yahoo.it}
\and Daniele Ritelli \footnote{Dipartimento scienze statistiche, via
Belle Arti, 41 40126 Bologna Italy, daniele.ritelli@unibo.it}}
\date{}
\maketitle
\begin{abstract}
In this article we developed a special topic of our pure-mathematics papers \cite{jnt3, jnt1, jnt2, jmaa1} concerning the hypergeometric theory.
Based upon a Roberts's reduction approach of hyperelliptic integrals to elliptic ones, \cite{rob}, and on the simultaneous multivariable hypergeometric series evaluation of them,  several identities have been obtained expressing $\pi$ in terms of special values of elliptic, hypergeometric and Gamma functions. 
By them $\pi$ can be provided through either only one or two parameters and through the imaginary unit.

In any case, such results, all unpublished and undoubtably new,  will provide, beyond their own beauty, a useful tool in order to check the routines (more or less na{\"i}ve) which one can build for the practical computations of Lauricella's functions met frequently in researches on Mechanics or Elasticity, for instance by ourselves in \cite{Laur1, Laur2, Laur3, Laur4}.

\end{abstract}
\section{Introduction}
First of all, a short framework of the Roberts's booklet, \cite{rob}, which is the start-up of all our next processing. After, a previous hypergeometric set of statements will be provided.
\subsection{A bit of history}
The british mathematician Michael Roberts (1817-1882) was attired by the hyperelliptic integrals since the publication of some relevant post-Jacobi papers by Riemann and Weierstrass. In 1871 he published a short tract, \cite{rob}, which collected his lectures on the subject.  He constructed a trigonometry of hyperelliptic functions analogous to the well-known one for the elliptic functions. Its sixth chapter deals with the {\it first class hyperelliptic integrals} depending on, or which can be reduced to, elliptic ones. 
They include for instance some integrals like:
\begin{equation}\label{robbgen}
 \int\frac{x^k}{\sqrt{Q_{g}(x)}}\,{\rm d}x
\end{equation}
where ${\it Q_{g}(x)}$ is a square-free polynomial of any given degree $g>4$. 
The allowable power {\it k} has to be determined by analysis of the possible pole at the point at infinity on the corresponding hyperelliptic curve. When this is done, one finds the condition to be $ k \leq g - 1,$ or, in other words, $k$ at most 1 for degree 5 or 6 of ${\it Q_{g}(x)}$; at most 2 for degrees 7 or 8, and so on.

Next Roberts gives the analogous with Fagnano's theorems founding upon the analogy between conic sections and the lines of curvature of an ellipsoid, linking to his best loved  research subject of geodesics on such a surface. A commemorative article, \cite{roberts}, issued  two years after his death, describes the scientific profile of M. Roberts.

In his booklet \cite{rob}, section 63, Roberts  takes into account reductions to elliptic of some special hyperelliptic integrals as:
\begin{equation}\label{robgen}
R_n=\int\frac{x^n}{\sqrt{x^8-p x^6+q x^4-p x^2+1}}\,{\rm d}x
\end{equation}
where $n$ takes the integral values 0, 2, 4. Reduction is carried out through the variable transformation:
\begin{equation}\label{cs}
f(x):=x+\frac{1}{x}=u\iff x=\frac12\left(u\pm\sqrt{u^2-4}\right)\implies{\rm d}x=\frac{1}{2}\left(1\pm\frac{u}{\sqrt{u^2-4}}\right).
\end{equation}
 The double sign is depending on being $f(x)$ not monotonic and  $+$ has to be taken when $x\in[0,1]$ and $-$ if $x\in(1,\infty).$
\begin{figure}[H]
\begin{center}
\scalebox{0.65}{\includegraphics{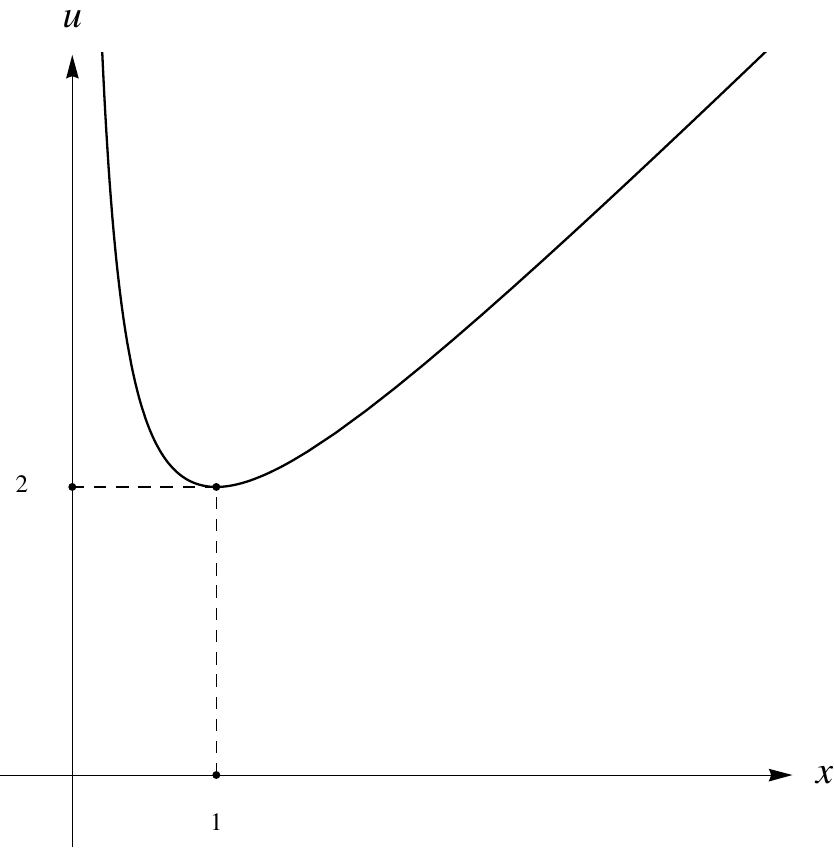}}\label{xf01} 
\end{center}
\caption{Cauchy-Schl{\"o}milch variable transformation}
\end{figure}
If the integration domain covers both the intervals $[0,1]$ and $[1,\infty)$, integration has to be properly split. After \eqref{cs} one finds:
\begin{equation}\label{cs1}
R_n=\frac{1}{2^{n-1}}\int\frac{\left(\sqrt{u^2-4}\pm u\right)^{n-1}}{\sqrt{u^2-4}}\frac{1}{\sqrt{u^4-(p+4)u^2+q+2p+2}}\,{\rm d}u.
\end{equation}

Such a transformation used by Roberts is known as  Cauchy-Schl{\"o}milch one, useful to evaluate many non-elementary definite integrals, most of which cannot be obtained by symbolic packages or otherwise. In \cite{Schlo} the \eqref{cs} is a tool for computing several of them and some historical notes are provided. In \cite{jnt3} \eqref{cs} is also extensively employed, following its use introduced by A.M. Legendre. In order to perform the necessary reduction, after the powering, one gets two integrals. The first holds the radical $\sqrt{u^2-4}$: it has also to be transformed by putting $v=\sqrt{u^2-4}$. Restricting for shortness to the case $x \in [0,1]$, the Roberts reductions are:\footnote{Notice that he is not concerned with defining exactly how the integration limits are transformed and neglects integration limits with the integrand $\geq1$.}
\begin{align}\label{rob0}
R_0&=\frac12\left(\int\frac{{\rm d}u}{\sqrt{u^4-(p+4) u^2+q+2 p+2}}+\int\frac{{\rm d}v}{\sqrt{v^4-(p-4) v^2+q-2 p+2}}\right)\\
R_2&=\frac12\left(-\int\frac{{\rm d}u}{\sqrt{u^4-(p+4) u^2+q+2 p+2}}+\int\frac{{\rm d}v}{\sqrt{v^4-(p-4) v^2+q-2 p+2}}\right)\label{rob2}\\
R_4&=\frac12\left(\int\frac{(1-u^2){\rm d}u}{\sqrt{u^4-(p+4) u^2+q+2 p+2}}+\int\frac{(1+v^2){\rm d}v}{\sqrt{v^4-(p-4) v^2+q-2 p+2}}\right)\label{rob4}
\end{align} 
where the left hand side hyperelliptic integrals have been reduced to a combination of elliptic ones with quartics under the quadratic radical.
We assumed to move to the expressions \eqref{rob0}, \eqref{rob2}, \eqref{rob4} because in $[0,1]$ the transformation \eqref{cs} is decreasing, so that, when passing to definite integrals, we will change the integration limits. Of course, when integrating in the interval $[1,\infty)$ the sign will be properly fit.
Through our method of double evaluation we already employed in \cite{jnt1, jnt2, jnt3, jmaa1}, we will establish further links among $\pi$, some multivariable hypergeometric functions and elliptic integrals. 

In order to treat hypergeometrically the hyperelliptic integral at left hand side of \eqref{robgen}, it will be necessary to solve the eighth-degree equation: 
\begin{equation}\label{nnomio}
P(x)=x^8-p x^6+q x^4-p x^2+1=0.
\end{equation} 
By the aforementioned \eqref{cs} we get:
\begin{equation}\label{nnomiobis}\tag{\ref{nnomio}b}
x^4\left(x^4+\frac{1}{x^4}-p\left(x^2+\frac{1}{x^2}\right)+q\right)=x^4\left(\left(u^2-2\right)^2-2-p\left(u^2-2\right)+q\right)
\end{equation}
where, due to \eqref{cs}, we put:
\[
x+\frac{1}{x}=u\implies x^2+\frac{1}{x^2}=u^2-2\implies x^4+\frac{1}{x^4}=\left(u^2-2\right)^2-2
\]
Then the roots of \eqref{nnomio} are found by solving the auxiliary equation:
\begin{equation}\label{nnomioter}\tag{\ref{nnomio}c}
u^4-(p+4)u^2+2 p+q+2=0,
\end{equation}
and then one goes back to originary variable $x$ via the \eqref{cs}. Due to the symmetry of equation \eqref{nnomio}, which is a reciprocal one, if it has the root $r,$ then $-r$ e $1/r$ will necessary be roots too. Three main cases can then occur:\begin{multicols}{2}
\begin{itemize}

\item 8  real roots

\item 4 real and 2 pairs of conjugate complex roots

\item 4 pairs of conjugate complex roots

\end{itemize}
\end{multicols}
Furthermore, if 1 is a root of \eqref{nnomio} it shall be double.

In our article, by comparing the hypergeometric approach for evaluating $R_n,$ see lemma \ref{iperlem}, lemma \ref{iperlemcom} and lemma \ref{iperlemcomcom}, versus the Roberts method, we will get not only several reductions of hyperelliptic integrals to elliptic ones, but also new formul\ae\, for $\pi$ to be added to those we presented in \cite{jnt1,jnt2,jnt3,jmaa1}. The treatment will depend on the index $n$ of the integral \eqref{robgen}: for $n=0,\,2,\,4$ the relevant complete integrals are studied taking as integration limits the roots of \eqref{nnomio}, or,  in the complex case, by integrating on the straight line $[0,\infty)$. Roberts's reductions are possible also for integrals with even further exponents and which can too be reduced to elliptic ones. But they have been omitted both for shortness and for avoid analogous developments but with much involved formul\ae. For instance, taking $n=6$ analogously to  \eqref{rob0}, \eqref{rob2} and \eqref{rob4} one arrives at:
\begin{align}
R_6&=\frac12\left(\int\frac{(-1+3u^2-u^4){\rm d}u}{\sqrt{u^4-(p+4) u^2+q+2 p+2}}-\int\frac{(1+3v^2+v^4){\rm d}v}{\sqrt{v^4-(p-4) v^2+q-2 p+2}}\right)\label{rob6}
\end{align}  
where the integrals at right hand side can be given as a combination of elliptic ones.
\section{Hypergeometric forewords}
In such a section some integration formul\ae  \    will be established which make use of multivariable hypergeometric functions and to be compared with Roberts reductions given in \cite{rob}.

By treating hypergeometrically the integrals like:
\begin{equation}\label{inteuno}
\int\frac{x^s}{\sqrt{(x^2-a^2)(x^2-b^2)(x^2-c^2)(x^2-d^2)}}\,{\rm d}x
\end{equation} several assumptions can be made:
 the numbers $a,\,b,\,c,\,d$ can be real and $0<a<b<c<d$ or $a,\,b$ real with $0<a<b$ and $c=\overline{d}$ or $a,\,b,\,c,\,d$ complex with $a=\overline{b},\,c=\overline{d}.$ When necessary, suitable conditons on $s$ will be required so that the integral \eqref{inteuno} succeeds to be convergent. Our first lemma \ref{iperlem} deals with $a,\,b,\,c,\,d$ all real, and further ones \ref{iperlemcom} and  \ref{iperlemcomcom} will deal with complex roots. The special functions necessary in order to evaluate the integrals \eqref{inteuno} are the hypergeometric Lauricella functions $F_{D}^{(n)}$ of $n\in \mathbb{N}^{+}$
variables, see \cite{saran1954} and \cite{lauricella1893}, defined as: 
\begin{equation*}
\mathrm{F}_{D}^{(n)}\left( \left. 
\begin{array}{c}
a,b_{1},\ldots ,b_{n} \\[2mm]
c
\end{array}
\right| x_{1},\ldots ,x_{n}\right):=  \sum_{m_{1},\ldots ,m_{n}\in \mathbb{N}}\frac{(a)_{m_{1}+\cdots
+m_{n}}(b_{1})_{m_{1}}\cdots (b_{n})_{m_{n}}}{(c)_{m_{1}+\cdots
+m_{n}}m_{1}!\cdots m_{n}!}\,x_{1}^{m_{1}}\cdots x_{n}^{m_{m}}
\end{equation*}
where $(x)_{k}$ is the Pochhammer symbol, and under the hypergeometric
series usual convergence requirements $|x_{1}|<1,\ldots ,|x_{n}|<1$. If $ 
\mathrm{Re}\,c>\mathrm{Re}\,a>0,$ the relevant integral representation theorem provides: 
\begin{equation}\label{IRT}
\mathrm{F}_{D}^{(n)}\left( \left. 
\begin{array}{c}
a,b_{1},\ldots ,b_{n} \\[2mm]
c
\end{array}
\right| x_{1},\ldots ,x_{n}\right) =\frac{
\Gamma (c)}{\Gamma (a)\,\Gamma (c-a)}\,\int_{0}^{1}\frac{%
u^{a-1}(1-u)^{c-a-1}\,{\rm d}u}{(1-x_{1}u)^{b_{1}}\cdots (1-x_{n}u)^{b_{n}}} 
\end{equation}
allowing the analytic continuation to $\mathbb{C}^{n}$ deprived of the
cartesian $n$-dimensional product of the interval $(1,\infty )$ with itself.

In order to define the hypergeometric integration, we take apart the different cases of roots of the eighth degree equation under the radical in \eqref{inteuno}.
\subsection{The case of 8 real roots}
We will restrict to $x>0$. In fact if $x\in[0,a]\cup[b,c]\cup[d,\infty)$ then $(x^2-a^2)(x^2-b^2)(x^2-c^2)(x^2-d^2)\geq0. $ We will see the following basic  definite integrals:
\begin{align}
J_1^{(s)}(a,b,c,d)&=\int_0^a\frac{x^s}{\sqrt{(x^2-a^2)(x^2-b^2)(x^2-c^2)(x^2-d^2)}}\,{\rm d}x\label{uno}\\
J_2^{(s)}(a,b,c,d)&=\int_b^c\frac{x^s}{\sqrt{(x^2-a^2)(x^2-b^2)(x^2-c^2)(x^2-d^2)}}\,{\rm d}x\label{due}\\
J_3^{(s)}(a,b,c,d)&=\int_d^\infty\frac{x^s}{\sqrt{(x^2-a^2)(x^2-b^2)(x^2-c^2)(x^2-d^2)}}\,{\rm d}x\label{tre}
\end{align}
By the next lemma \ref{iperlem}, we will be allowed to compute \eqref{uno}, \eqref{due} and \eqref{tre} hypergeometrically:
\begin{lemma}\label{iperlem}
If $0<a<b<c<d,$ then:
\begin{align}
J_1^{(s)}(a,b,c,d)&=\label{unot}\frac{a^s}{2bcd}\frac{\sqrt{\pi}\,\Gamma\left(\frac{s+1}{2}\right)}{\Gamma\left(\frac{s+2}{2}\right)}\,\mathrm{F}_{D}^{(3)}\left( \left. 
\begin{array}{c}
\frac{1+s}{2};\frac12,\frac12,\frac12 \\[2mm]
\frac{s+2}{2}
\end{array}
\right| \frac{a^2}{b^2},\frac{a^2}{c^2},\frac{a^2}{d^2}\right)\tag{\ref{uno}b}\\
J_2^{(s)}(a,b,c,d)&=\frac{\pi b^{s-1}}{2\sqrt{\left(b^2-a^2\right) \left(d^2-b^2\right)}}\,\mathrm{F}_{D}^{(3)}\left( \left. 
\begin{array}{c}
\frac12;\frac{1-s}{2},\frac12,\frac12 \\[2mm]
1
\end{array}
\right| -\frac{c^2-b^2}{b^2},-\frac{c^2-b^2}{b^2-a^2},\frac{c^2-b^2}{d^2-b
   ^2}\right)\label{duet}\tag{\ref{due}b}\\
J_3^{(s)}(a,b,c,d)&=\frac{d^s}{2 \sqrt{\left(d^2-a^2\right) \left(d^2-b^2\right)
   \left(d^2-c^2\right)}}\frac{\sqrt{\pi}\,\Gamma\left(\frac{3-s}{2}\right)}{\Gamma\left(2-\frac{s}{2}\right)}\,\mathrm{F}_{D}^{(3)}\left( \left. 
\begin{array}{c}
\frac12;\frac12,\frac12,\frac12 \\[2mm]
2
\end{array}
\right| \frac{a^2}{a^2-d^2},\frac{b^2}{b^2-d^2},\frac{c^2}{c^2-d^2}\right)\label{tret}\tag{\ref{tre}b}
\end{align}
where in \eqref{uno} and \eqref{unot}we assume $s>-1$ and in \eqref{tre} and \eqref{tret} we assume $s<3$.
\end{lemma}
\begin{proof}
 Through the integral representation theorem and our customary approach as in papers \cite{jnt1, jnt2, jnt3,jmaa1}, we will compute the integrals \eqref{uno}, \eqref{due}, \eqref{tre}. 
 Formula \eqref{unot} can be obtained via the change $x=a\sqrt{u}$ in \eqref{uno} and calling \eqref{IRT}. Analogously \eqref{duet} is found putting $x=\sqrt{b^2+(c^2-b^2)u},$ whilst \eqref{tret} follows from $x=d/\sqrt{1-u}$.
\end{proof}
\subsection{The case of 4 real roots and 2 pairs of conjugate complex roots
}
The following lemma \ref{iperlemcom} will concern the integrand with complex roots.
\begin{lemma}\label{iperlemcom}
Let it be $a,\,b\in\mathbb{R},\,a<b,\,\,z\in\mathbb{C}\setminus\mathbb{R}.$ We will consider the integrals:
\begin{align}
M_1(a,b,z)&=\int_0^a\frac{x^s}{\sqrt{(x^2-a^2)(x^2-b^2)(x^2-z^2)(x^2-\overline{z}^2)}}\,{\rm d}x\label{uunobb}\\
M_2(a,b,z)&=\int_b^\infty\frac{x^s}{\sqrt{(x^2-a^2)(x^2-b^2)(x^2-z^2)(x^2-\overline{z}^2)}}\,{\rm d}x\label{duebb}
\end{align}
where in \eqref{uunobb} we assume $s>-1$ amd in \eqref{duebb}  $s<3.$
Then:
\begin{align}
M_1(a,b,z)&=\frac{a^s\sqrt{\pi}}{2b\left|z\right|^2}\frac{\Gamma\left(\frac{s+1}{2}\right)}{\Gamma\left(\frac{s+2}{2}\right)}\label{unobbhy}\tag{\ref{uunobb}b}\,\mathrm{F}_{D}^{(3)}\left( \left. 
\begin{array}{c}
\frac{s+1}{2};\frac12,\frac12,\frac12 \\[2mm]
\frac{s+2}{2}
\end{array}
\right| \frac{a^2}{b^2},\frac{a^2}{z^2},\frac{a^2}{\overline{z}^2}\right)\\
M_2(a,b,z)&=\frac{b^s\sqrt{\pi}}{2\sqrt{(b^2-a^2)(b^2-z^2)(b^2-\overline{z}^2)}}\,\frac{\Gamma\left(\frac32-\frac{s}{2}\right)}{\Gamma\left(2-\frac{s}{2}\right)}\,\mathrm{F}_{D}^{(3)}\left( \left. 
\begin{array}{c}
\frac12;\frac12,\frac12,\frac12 \\[2mm]
\frac{4-s}{2}
\end{array}
\right| -\frac{a^2}{b^2-a^2},-\frac{z^2}{b^2-z^2},-\frac{\overline{z}^2}{b^2-\overline{z}^2}\right)\label{duebbhy}\tag{\ref{duebb}b}
\end{align}

\end{lemma}
\begin{proof}
On the steps of proof of lemma \ref{iperlem} in the integral \eqref{unobb} the change $x=a\sqrt{u}$ is used, whilst for \eqref{duebb} the change $x=b(1-u)^{-1/2}.$ Thesis \eqref{unobbhy} and \eqref{duebbhy} follow from the integral representation theorem. 
\end{proof}
\subsection{The case of 4 pairs of conjugate complex roots
}
Finally, analogous of two preceding lemmas \ref{iperlem}, \ref{iperlemcom} when the integrand has all its roots complex:
\begin{lemma}\label{iperlemcomcom}
Let $a,\,b\in\mathbb{C}\setminus\mathbb{R}$ then, if $s<3$
\begin{equation}\label{allcomp}
\begin{split}
L(a,b)&=\int_0^\infty\frac{x^s}{\sqrt{(x^2-a^2)(x^2-\overline{a}^2)(x^2-b^2)(x^2-\overline{b}^2)}}\,{\rm d}x\\
&=\frac{\pi(1-s)}{4\cos\left(\frac{\pi s}{2}\right)}\mathrm{F}_{D}^{(4)}\left( \left. 
\begin{array}{c}
\frac{3-s}{2};\frac12,\frac12,\frac12,\frac12 \\[2mm]
2
\end{array}
\right| 1+a^2,1+\overline{a}^2,1+b^2,1+\overline{b}^2\right)
\end{split}
\end{equation}
\end{lemma}
\begin{proof}
In order to apply \eqref{IRT}, it will be enough a change in \eqref{allcomp} putting first $x=\sqrt{v}$ and after $v=(1-u)/u$ in such a way obtaining the integral on $[0,1]$: 
\[
\frac12\int_0^1\frac{u^{\frac{1-s}{2}}(1-u)^{-\frac{1-s}{2}}}{\sqrt{1-\left(1+a^2\right) u} \sqrt{1-\left(1+b^2\right) u}
   \sqrt{1-\left(1+c^2\right) u} \sqrt{1-\left(1+d^2\right) u}}\,{\rm d}u.
\]
\end{proof}

\section{The case $\boldsymbol{n=0}$}

We want to study the consequences of assuming the values $n= 0, 2, 4$ in \eqref{robgen} and for each of them the relevant sub-cases.
At each $n$-value will be devoted a separate section. We start with $n=0$.


\subsection{Sub-case: 8 real roots}
In order to use the Roberts transformation, due to the $P(x)$ structure, we will have two numbers $a,\,b$ such that:
\begin{equation}\label{effereale}
P(x)=(x^2-a^2)(x^2-\frac{1}{a^2})(x^2-b^2)(x^2-\frac{1}{b^2})
\end{equation}
Then in the \eqref{robgen} integrand, we have:
\begin{equation}\label{piqu}
p=\frac{\left(a^2+b^2\right) \left(1+a^2 b^2\right)}{a^2 b^2},\quad q=\frac{1+a^4 b^4+\left(a^2+b^2\right)^2}{a^2 b^2}.
\end{equation}
Fixing two real numbers  $1<a<b$, by \eqref{robgen} we get complete elliptic integrals to be integrated in $[0,1/b],\,[1/a,a],$ $[b,\infty),$ when $P(x)$ has the form \eqref{effereale}. Let us see them apart. 

\subsubsection*{Interval $\boldsymbol{[0,1/b]}$}
We introduce a special symbol for describing how \eqref{robgen} specializes taking into account: the index $n,$ the integration interval, and the number of $P(x)$'s real roots. We put:
\begin{equation}
_8R_n^{[0,1/b]}(a,b)=\int_0^{\frac{1}{b}}\frac{x^n{\rm d}x}{\sqrt{(x^2-a^2)(x^2-\frac{1}{a^2})(x^2-b^2)(x^2-\frac{1}{b^2})}}
\end{equation}
The left index \lq\lq8'' recalls the number of real roots of \eqref{nnomio}; the apex concerns the integration interval; the right index marks the degree of power at numerator of \eqref{robgen}, and finally the argument within round brackets defines the generators of the roots of \eqref{nnomio}.
If $n=0$ we have:
\begin{teorema}
With $1<a<b$ we get:
\begin{equation}
_8R_0^{[0,1/b]}(a,b)=\frac{1}{2}\left(\frac{b}{1+b^2}\boldsymbol{K}\left(\frac{b(1+a^2)}{a(1+b^2)}\right)-\frac{b}{1-b^2}\boldsymbol{K}\left(\frac{b(1-a^2)}{a(1-b^2)}\right)\right)\label{63realeB}
\end{equation}
\end{teorema}
\begin{proof}

Let us compute the last above integral by the  Roberts reduction formula \eqref{rob0}, through \eqref{piqu}. We get:
\begin{equation}\label{63reale}
_8R_0^{[0,1/b]}(a,b)=\frac12\left(\int_{b+\frac{1}{b}}^\infty\frac{{\rm d}u}{\sqrt{\left[\left(a+\frac{1}{a}\right)^2-u^2\right]\left[\left(b+\frac{1}{b}\right)^2-u^2\right]}}+\int_{b-\frac{1}{b}}^\infty\frac{{\rm d}u}{\sqrt{\left[\left(a-\frac{1}{a}\right)^2-u^2\right]\left[\left(b-\frac{1}{b}\right)^2-u^2\right]}}\right)
\end{equation}
Both integrals at right side of \eqref{63reale} are complete elliptic of first kind:
\begin{equation}\label{63realetoo}
\int_\alpha^\infty\frac{{\rm d}x}{\sqrt{(\alpha^2-x^2)(\beta^2-x^2)}}=\frac{1}{\alpha}\boldsymbol{K}\left(\frac{\beta}{\alpha}\right),\quad\beta<\alpha,
\end{equation}
see entries 3.151-12 p. 277 of \cite{gr} and 215.00 of \cite{by}. Thus thesis follows.
\end{proof}

By means of lemma \ref{iperlem} we can get our further formula providing $\pi$ and to be added to other ours given in \cite{jnt1,jnt2,jnt3,jmaa1}. 

\begin{teorema}\label{1uno}
If $1<a<b$ then
\begin{equation}\label{tesiA}
\pi=\frac{\displaystyle{\frac{b^2}{b^2+1}\boldsymbol{K}\left(\frac{\left(a^2+1\right) b}{\left(b^2+1\right)a}\right)+\frac{b^2}{b^2-1}\boldsymbol{K}\left(\frac{\left(a^2-1\right) b}{\left(b^2-1\right)a}\right)}}{\displaystyle{\mathrm{F}_{D}^{(3)}\left( \left. 
\begin{array}{c}
\frac12;\frac12,\frac12,\frac12 \\[2mm]
1
\end{array}
\right| \frac{a^2}{b^2},\frac{1}{a^2b^2},\frac{1}{b^4}\right)}}
\end{equation}
\end{teorema}
\begin{proof}
We use the first integration formula of lemma \ref{iperlem}, for a integral like \eqref{uno}. The roots identification is:
\[
\begin{matrix} a & b & c & d \\ \updownarrow & \updownarrow &
\updownarrow & \updownarrow \\ 1/b & 1/a & a & b \end{matrix} 
\]
whence the hypergeometric integration formula true for $1<a<b$
\begin{equation}\label{iperuno}
_8R_0^{[0,1/b]}(a,b)=\frac{\pi}{2b}\mathrm{F}_{D}^{(3)}\left( \left. 
\begin{array}{c}
\frac12;\frac12,\frac12,\frac12 \\[2mm]
1
\end{array}
\right| \frac{a^2}{b^2},\frac{1}{a^2b^2},\frac{1}{b^4}\right)
\end{equation}

Thesis \eqref{tesiA} follows by reducing the right hand side of \eqref{63reale}, applying the \eqref{iperuno} for computing $_8R_0^{[0,1/b]}(a,b)$ hypergeometrically and then getting $\pi.$
\end{proof}
When one of $P(x)$ real roots has the value 1, such a root has to be double according to the order reduction between the three variable hypergeometric function ${\rm F}_D^{(3)}$ and the two variable one named ${\rm F}_1$: 
\begin{equation}\label{riduz}
\mathrm{F}_{D}^{(3)}\left( \left. 
\begin{array}{c}
a;b_1,b_2,b_3 \\[2mm]
c
\end{array}
\right| x,x,z\right)=\mathrm{F}_1\left( \left. 
\begin{array}{c}
a;b_1+b_2,b_3 \\[2mm]
c
\end{array}
\right| x,z\right)
\end{equation}
and also due to the relationship:
\begin{equation}\label{elementare}
\boldsymbol{K}(0)=\int_0^1\frac{{\rm d}x}{\sqrt{(1-x^2)}}=\frac{\pi}{2}
\end{equation}
So that one is allowed to establish:
\begin{corollario}\label{cor1uno}
If $b>1$ then
\begin{equation}\label{cor1uno1}
\pi=\frac{2 b^2 \left(b^2-1\right)\boldsymbol{K}\left(\frac{2b}{1+b^2}\right) }{\left(1+b^2\right) \left(2\left(b^2-1\right)\mathrm{F}_1\left( \left. 
\begin{array}{c}
\frac12;1,\frac12 \\[2mm]
1
\end{array}
\right| \frac{1}{b^2},\frac{1}{b^4}\right)
   -b^2\right)}
\end{equation}
\end{corollario}
\subsubsection*{Interval $\boldsymbol{[1/a,a]}$}
Now let us consider the complete integral on $[1/a,a]:$ 
\begin{equation}\label{iperdue}
_8R_0^{[1/a,a]}(a,b):=\int_{\frac{1}{a}}^a\frac{{\rm d}x}{\sqrt{(x^2-a^2)(x^2-\frac{1}{a^2})(x^2-b^2)(x^2-\frac{1}{b^2})}}
\end{equation}
By the trasformation \eqref{cs} we can establish the following integration formula:
\begin{teorema}\label{ttdue}
Let it be $1<a<b$, then
\begin{equation}\label{kompl2}
_8R_0^{[1/a,a]}(a,b)=\frac{b}{b^2-1}\boldsymbol{K}\left(\frac{b\left(a^2-1\right)}{a \left(b^2-1\right)}\right).
\end{equation}

\end{teorema}
\begin{proof}
Integral $_8R_0^{[1/a,a]}(a,b)$ will now be computed via transformation \eqref{cs}: for the purpose, the integration interval has to be split:
\[
\begin{split}
_8R_0^{[1/a,a]}(a,b)&=\int_{\frac{1}{a}}^1\frac{{\rm d}x}{\sqrt{(x^2-a^2)(x^2-\frac{1}{a^2})(x^2-b^2)(x^2-\frac{1}{b^2})}}+\int_{1}^a\frac{{\rm d}x}{\sqrt{(x^2-a^2)(x^2-\frac{1}{a^2})(x^2-b^2)(x^2-\frac{1}{b^2})}}\\
&=I_2^{(1)}(a,b)+I_2^{(2)}(a,b)
\end{split}
\]
For $I_2^{(1)}(a,b)$ one goes on with transformation \eqref{cs}, namely, moving there $x$  within $[0,1]$, the root
\[
x=\frac12\left(u-\sqrt{u^2-4}\right).
\]
has to be chosen. Therefore, after changing the integration limits, we get:
\[
I_2^{(1)}(a,b)=\frac12\left(\int_2^{a+\frac{1}{a}}\frac{{\rm d}u}{\sqrt{\left[\left(a+\frac{1}{a}\right)^2-u^2\right]\left[\left(b+\frac{1}{b}\right)^2-u^2\right]}}+\int_0^{a-\frac{1}{a}}\frac{{\rm d}u}{\sqrt{\left[\left(a-\frac{1}{a}\right)^2-u^2\right]\left[\left(b-\frac{1}{b}\right)^2-u^2\right]}}\right)
\]
and through formul\ae  \   3.152-8 and 3.152-7 p. 276 of \cite{gr} or  entries 220.00 and 219.00 from\cite{by}:
\[
I_2^{(1)}(a,b)=\frac12\left(\frac{b}{b^2+1}F\left(\arcsin\frac{\left(a^2-1\right) \left(b^2+1\right)}{\left(a^2+1\right)
   \left(b^2-1\right)},\frac{b\left(a^2+1\right)}{a \left(b^2+1\right)}\right)+\frac{b}{b^2-1}\boldsymbol{K}\left(\frac{b\left(a^2-1\right)}{a \left(b^2-1\right)}\right)\right)
\]
For $I_2^{(2)}(a,b)$ the approach is similar: through the transformation  $x+1/x=u$ with $x$ within $[1,a]$ with $a>1$, so that the root to be chosen is:
\[
x=\frac12\left(u+\sqrt{u^2-4}\right)
\]
As a consequence:
\[
\begin{split}
I_2^{(2)}(a,b)&=\frac12\left(-\int_2^{a+\frac{1}{a}}\frac{{\rm d}u}{\sqrt{\left[\left(a+\frac{1}{a}\right)^2-u^2\right]\left[\left(b+\frac{1}{b}\right)^2-u^2\right]}}+\int_0^{a-\frac{1}{a}}\frac{{\rm d}u}{\sqrt{\left[\left(a-\frac{1}{a}\right)^2-u^2\right]\left[\left(b-\frac{1}{b}\right)^2-u^2\right]}}\right)\\
&=\frac12\left(-\frac{b}{b^2+1}F\left(\arcsin\frac{\left(a^2-1\right) \left(b^2+1\right)}{\left(a^2+1\right)
   \left(b^2-1\right)},\frac{b\left(a^2+1\right)}{a \left(b^2+1\right)}\right)+\frac{b}{b^2-1}\boldsymbol{K}\left(\frac{b\left(a^2-1\right)}{a \left(b^2-1\right)}\right)\right)
\end{split}
\]
Thesis \eqref{kompl2} follows by adding the expressions obtained of $I_1^{(2)}(a,b)$ and $I_2^{(2)}(a,b)$.

\end{proof}
As a consequence of theorem \ref{ttdue} just proved, by lemma \ref{iperlem} again, equation \eqref{duet}, we obtain the second $\pi$ formula of this paper:
\begin{teorema}\label{2due}
If $1<a<b$ :
\begin{equation}\label{pibis}
\pi=\frac{2 \sqrt{\left(b^2-a^2\right) \left(a^2 b^2-1\right)}}{a^3 \left(b^2-1\right)}\,\frac{\displaystyle{\boldsymbol{K}\left(\frac{b\left(a^2-1\right)}{a \left(b^2-1\right)}\right)}}{\displaystyle{\mathrm{F}_{D}^{(3)}\left( \left. 
\begin{array}{c}
\frac12;\frac12,\frac12,\frac12 \\[2mm]
1
\end{array}
\right| 1-a^4,\frac{(1-a^4)b^2}{b^2-a^2},\frac{a^4-1}{a^2b^2-1}\right)}}
\end{equation}
\end{teorema}
\begin{proof}
Through lemma \ref{iperlem} equation \eqref{duet} we compute the integral \eqref{iperdue}
so that:
\begin{equation}\label{iperdueb}\tag{\ref{iperdue}b}
_8R_0^{[1/a,a]}(a,b)=\frac{a^3b\pi}{2\sqrt{(b^2-a^2)(a^2b^2-1)}}\,\mathrm{F}_{D}^{(3)}\left( \left. 
\begin{array}{c}
\frac12;\frac12,\frac12,\frac12 \\[2mm]
1
\end{array}
\right| 1-a^4,\frac{(1-a^4)b^2}{b^2-a^2},\frac{a^4-1}{a^2b^2-1}\right)
\end{equation}
By comparing \eqref{iperdueb} with \eqref{kompl2}, thesis \eqref{pibis} follows.
\end{proof}
\subsubsection*{Interval $\boldsymbol{[b,\infty)}$}

Now let us deal with the complete integral along $[b,\infty):$
\begin{equation}
_8R_0^{[b,\infty)}(a,b):=\int_{b}^\infty\frac{{\rm d}x}{\sqrt{(x^2-a^2)(x^2-\frac{1}{a^2})(x^2-b^2)(x^2-\frac{1}{b^2})}}\label{ipertre}
\end{equation}
\begin{teorema}\label{ttrett}
If $1<a<b$ then
\begin{equation}\label{tess33}
_8R_0^{[b,\infty)}(a,b)=\frac12\left(\frac{b}{b^2-1}\,\boldsymbol{K}\left(\frac{b(a^2-1)}{a(b^2-1)}\right)-\frac{b}{b^2+1}\,\boldsymbol{K}\left(\frac{b(a^2+1)}{a(b^2+1)}\right)\right)
\end{equation}
\end{teorema}
\begin{proof}
By the Roberts reduction again, being the region where the transformation \eqref{cs} is growing, we find:
\begin{equation}\label{riduter}
_8R_0^{[b,\infty)}(a,b)=\frac12\left(\int_{b-\frac{1}{b}}^\infty\frac{{\rm d}u}{\sqrt{\left[\left(a-\frac{1}{a}\right)^2-u^2\right]\left[\left(b-\frac{1}{b}\right)^2-u^2\right]}}-\int_{b+\frac{1}{b}}^\infty\frac{{\rm d}u}{\sqrt{\left[\left(a+\frac{1}{a}\right)^2-u^2\right]\left[\left(b+\frac{1}{b}\right)^2-u^2\right]}}\right)
\end{equation}
Then thesis \eqref{tess33} follows by computing the complete elliptic integrals in \eqref{riduter} through the entries 3.153-11/12 p. 277 of \cite{gr} or 215.00 and 216.00 of \cite{by}.
\end{proof}
 From lemma \ref{iperlem}, through  \eqref{tret} for computing the integral \eqref{ipertre} and replacing the outcome with the theorem \ref{ttrett}, formula for $\pi$ follows:
\begin{teorema}\label{3tre}
Let $a<b<1$, then:
\begin{equation}\label{piter}
\pi=\frac{2}{a}\sqrt{\frac{(b^2-a^2)(1-a^2 b^2)}{1-b^4}}\,\frac{\displaystyle{\left(1+b^2\right) \boldsymbol{K}\left(\frac{b\left(a^2-1\right)}{a
   \left(b^2-1\right)}\right)+\left(1-b^2\right) \boldsymbol{K}\left(\frac{b(a^2+1)}{a(b^2+1)}\right)}}{\displaystyle{\mathrm{F}_{D}^{(3)}\left( \left. 
\begin{array}{c}
\frac12;\frac12,\frac12,\frac12 \\[2mm]
2
\end{array}
\right| \frac{1}{1-b^4},\frac{1}{1-a^2 b^2},\frac{a^2}{a^2-b^2}\right)}}
\end{equation}
\end{teorema}
\begin{proof}
By lemma \ref{iperlem}, equation \eqref{tret} we have:
\begin{equation}\label{iper3tre}
\begin{split}
_8R_0^{[b,\infty)}(a,b)&=\int_{b}^\infty\frac{{\rm d}x}{\sqrt{(x^2-a^2)(x^2-b^2)(x^2-\frac{1}{a^2})(x^2-\frac{1}{b^2})}}\\
&=\frac{\pi  a b}{4}\frac{1}{\sqrt{\left(b^4-1\right) \left(b^2-a^2\right) \left(a^2 b^2-1\right)}}\mathrm{F}_{D}^{(3)}\left( \left. 
\begin{array}{c}
\frac12;\frac12,\frac12,\frac12 \\[2mm]
2
\end{array}
\right| \frac{1}{1-b^4},\frac{1}{1-a^2 b^2},\frac{a^2}{a^2-b^2}\right)
\end{split}
\end{equation}
Our further  $\pi$ formula follows after equating to \eqref{iper3tre} and comparing to \eqref{tess33}.
\end{proof}
Taking the limit for $a\to1,$ a formula to $\pi$ comes out characterized by only one parameter.
\begin{corollario}\label{kkoro3}
If $b>1$ then
\begin{equation}\label{koro3}
\pi=\frac{2 \left(b^2-1\right)^2 }{\sqrt{b^4-1}
   }\frac{\displaystyle{\boldsymbol{K}\left(\frac{2
   b}{1+b^2}\right)}}{\displaystyle{\sqrt{b^4-1}-\mathrm{F}_1\left( \left. 
\begin{array}{c}
\frac12;1,\frac12 \\[2mm]
2
\end{array}
\right| \frac{1}{1-b^2},\frac{1}{1-b^4}\right)}}
\end{equation}
\end{corollario}
\subsection{Sub-case of 4 real and 2 pairs of conjugate complex roots}
Let the polynomial \eqref{nnomio} have 4 real roots and 2 pairs of conjugate complex ones, assumed to belong to the unity disk. Then we will have with $a,\,b\in\mathbb{R},\,b>1$:
\begin{equation}\label{nnomio2compl}
\begin{split}
P(x)&=\left(x^2-\frac{1}{b^2}\right) \left(x^2-b^2\right) \left(x^2-e^{-2 i a}\right)
   \left(x^2-e^{2 i a}\right)\\
   &=\left(x^2-\frac{1}{b^2}\right) \left(x^2-b^2\right) \left(1-2 \cos2a\,x^2 +x^4\right)\\
   &=1-\frac{1+2 b^2 \cos2a+b^4}{b^2}x^2+2\frac{1+2 b^2 \cos2a+b^4}{b^2}x^4-\frac{1+2 b^2 \cos2a+b^4}{b^2}x^6+x^8
   \end{split}
\end{equation}
Restricting to the positive half straight line, if $x\in[0,1/b]\cup[b,\infty)$ we have $P(x)\geq0$. We pass to compute the integrals:
\begin{align}
_4R_0^{[0,1/b]}(a,b)&=\int_0^{\frac{1}{b}}\frac{{\rm d}x}{\sqrt{\left(x^2-\frac{1}{b^2}\right) \left(x^2-b^2\right) \left(x^2-e^{-2 i a}\right)
   \left(x^2-e^{2 i a}\right)}}\label{unob}\\
_4R_0^{[b,\infty)}(a,b)&=\int_b^\infty\frac{{\rm d}x}{\sqrt{\left(x^2-\frac{1}{b^2}\right) \left(x^2-b^2\right) \left(x^2-e^{-2 i a}\right)
   \left(x^2-e^{2 i a}\right)}}\label{dueb}
\end{align}
by comparing the hypergeometric evaluation of \eqref{unob} and \eqref{dueb}, with those obtained as complete elliptic integrals through the trasformation \eqref{cs}. 

Through lemma \ref{iperlemcom} we get two further formul\ae\,  for $\pi$, we state the following theorems  \ref{1picompluno} and \ref{2picompluno}:
\subsubsection*{Interval $\boldsymbol{[0,1/b]}$}
\begin{teorema}\label{1picompluno}
If $a,\,b\in\mathbb{R},\,b>1$ then:
\begin{equation}\label{ipcomlunotesi}
\pi=\frac{b^2}{{\displaystyle{\mathrm{F}_{D}^{(3)}\left( \left. 
\begin{array}{c}
\frac12;\frac12,\frac12,\frac12 \\[2mm]
1
\end{array}
\right| \frac{1}{b^4},\frac{e^{-2ia}}{b^2},\frac{e^{2ia}}{b^2}\right)}}}\left(\frac{\displaystyle{\boldsymbol{K}\left(\frac{2 b \cos a}{1+b^2}\right)}}{1+b^2}+\frac{\displaystyle{\boldsymbol{K}\left(\frac{2 b \sin a}{\sqrt{1-2b^2  \cos2a+b^4}}\right)}}{\sqrt{1-2b^2\cos2a+b^4}}\right).
\end{equation}
\end{teorema}
\begin{proof}
Through the formula \eqref{unobbhy} we get the hypergeometric evaluation of $_4R_0^{[0,1/b]}(a,b)$ put into \eqref{unob}:
\begin{equation}\label{unobb}\tag{\ref{unob}a}
_4R_0^{[0,1/b]}(a,b)=\frac{\pi}{2b}\,\mathrm{F}_{D}^{(3)}\left( \left. 
\begin{array}{c}
\frac12;\frac12,\frac12,\frac12 \\[2mm]
1
\end{array}
\right| \frac{1}{b^4},\frac{e^{-2ia}}{b^2},\frac{e^{2ia}}{b^2}\right).
\end{equation}
Furthermore, by transformation \eqref{cs} to \eqref{unob}, through arguments like those seen for theorem \ref{1uno}, we get:
\begin{equation}\label{unobc}\tag{\ref{unob}b}
_4R_0^{[0,1/b]}(a,b)=\frac12\left(\int_{b+\frac{1}{b}}^\infty\frac{{\rm d}u}{\sqrt{\left(\left(b+\frac{1}{b}\right)^2-u^2\right) \left(4 \cos^2a-u^2\right)}}+\int_{b-\frac{1}{b}}^\infty\frac{{\rm d}u}{\sqrt{\left(u^2-\left(b-\frac{1}{b}\right)^2\right) \left(4 \sin^2a+u^2\right)}}\right).
\end{equation}
Both elliptic integrals at right hand side of \eqref{unobc} can be evaluated by entries 3.152-12 and 3.152-6 p. 277 of \cite{gr} so that:
\begin{equation}\label{unobd}\tag{\ref{unob}c}
_4R_0^{[0,1/b]}(a,b)=\frac12\left(\frac{b}{1+b^2}\boldsymbol{K}\left(\frac{2 b \cos a}{1+b^2}\right)+\frac{b}{{\sqrt{1-2b^2\cos2a+b^4}}}\boldsymbol{K}\left(\frac{2 b \sin a}{\sqrt{1-2  b^2\cos2a+b^4}}\right)\right)
\end{equation}
Thesis \eqref{ipcomlunotesi} comes by equating  \eqref{unobb} with \eqref{unobd} and then expressing $\pi$ via $a,\, b$ and the imaginary unit. 
\end{proof}
\subsubsection*{Interval $\boldsymbol{[b,\infty)}$}
\begin{teorema}\label{2picompluno}
If $a,\,b\in\mathbb{R},\,b>1$ then
\begin{equation}\label{ipcom2duetesi}
\pi=\frac{2 \sqrt{b^4-1} \sqrt{1-2\cos2a\, b^2 +b^4}}{{\displaystyle{\mathrm{F}_{D}^{(3)}\left( \left. 
\begin{array}{c}
\frac12;\frac12,\frac12,\frac12 \\[2mm]
2
\end{array}
\right|\frac{1}{1-b^4},\frac{1}{1-b^2e^{-2 i a}},\frac{1}{1-b^2e^{2 i a} }\right)}}}\left(\frac{\displaystyle{\boldsymbol{K}\left(\frac{2 b \sin a}{\sqrt{1-2\cos2a\, b^2 +b^4}}\right)}}{\sqrt{1-2\cos2a\, b^2 +b^4}}-\frac{\displaystyle{\boldsymbol{K}\left(\frac{2 b \cos a}{1+b^2}\right)}}{1+b^2}\right)
\end{equation}
\end{teorema}
\begin{proof}
By the hypergeometrical side we compute the integral \eqref{dueb} by means of lemma \ref{iperlemcom} eq. \eqref{duebb}, obtaining
\begin{equation}\label{duebbb}
_4R_0^{[b,\infty)}(a,b)=\frac{b\pi}{4 \sqrt{b^4-1} \sqrt{1-2 \cos2a\,b^2 +b^4}}\mathrm{F}_{D}^{(3)}\left( \left. 
\begin{array}{c}
\frac12;\frac12,\frac12,\frac12 \\[2mm]
2
\end{array}
\right|\frac{1}{1-b^4},\frac{1}{1-b^2e^{-2 i a}},\frac{1}{1-b^2e^{2 i a} }\right)
\end{equation}
On the other side, the Cauchy-Schl\"omilch transformation provides:
\begin{equation}\label{duebbbb}
_4R_0^{[b,\infty)}(a,b)=\frac12\left(\int_{b-\frac{1}{b}}^\infty\frac{{\rm d}u}{\sqrt{\left(u^2-\left(b-\frac{1}{b}\right)^2\right) \left(4\sin^2a+u^2\right)}}-\int_{b+\frac{1}{b}}^\infty \frac{{\rm d}u}{\sqrt{\left(\left(b+\frac{1}{b}\right)^2-u^2\right) \left(4\cos^2a-u^2\right)}}\right)
\end{equation}
The above elliptic integrals in \eqref{duebbb} are listed in \cite{gr} entry 3.152 - 12 p. 277 so that
\begin{equation}\label{duebbbbb}
_4R_0^{[b,\infty)}(a,b)=\frac{1}{2}\left(\frac{b}{{\sqrt{1-2b^2\cos2a +b^4}}}\boldsymbol{K}\left(\frac{2 b \sin a}{\sqrt{1-2b^2\cos2a +b^4}}\right)-\frac{b}{1+b^2}\boldsymbol{K}\left(\frac{2b\cos a}{1+b^2}\right)\right)
\end{equation}
Thesis \eqref{ipcom2duetesi} follows by equating \eqref{duebbb} and \eqref{duebbbbb} and getting $\pi$.
\end{proof}
\begin{osservazione}
Notice that it is not possible, in statements of theorems \ref{1picompluno} and \ref{2picompluno}, to make $b\to1^{+}$ because both integrals \eqref{unob} and \eqref{dueb} are divergent if $b\to1^{+}$.
\end{osservazione}
\subsection{Sub-case of 4 pairs of conjugate complex roots}
Even if the polynomial \eqref{nnomio} does not have any real root, we can refer to a representation formula to $\pi,$ theorem  \ref{picomcom}. Let all the $P(x)$'s roots stay within the unity disk. This assumption, joined to the special $P(x)$ structure will imply:
\begin{equation}\label{complexx}
\begin{split}
P(x)&=(x^2-e^{2ia})(x^2-e^{-2ia})(x^2-e^{2ib})(x^2-e^{-2ib})\\
&=\left(1-2 \cos2a\,x^2 +x^4\right) \left(1-2  \cos2b\,x^2 +x^4\right)\\
&=1-2  (\cos2a+\cos2b)\,x^2+ (2+4\cos2a \cos2b)\,x^4-2 (\cos2a+\cos2b)\, x^6+x^8
\end{split}
\end{equation}
where $a,\,b$ are two real numbers upon which later we will do some further assumptions.

First, again, an hyperelliptic integral computed hypergeometrically: by means of lemma \ref{iperlemcomcom} this will be done for the integral: 
\begin{equation}\label{intcomcom}
_0R_0^{[0,\infty)}(a,b)=\int_0^\infty\frac{{\rm d}x}{\sqrt{\left(1-2 \cos2a\,x^2 +x^4\right) \left(1-2  \cos2b\,x^2 +x^4\right)}}.
\end{equation}
\subsubsection*{Interval $\boldsymbol{[0,\infty)}$}
Following Roberts, applying again the theorem \ref{2due}, then \eqref{intcomcom} is promptly seen as a first kind complete elliptic integral: then by the usual comparison of different reckoning, we get:
\begin{teorema}\label{picomcom}
Let $a,\,b\in\mathbb{R}$ be such that $\sin a>\sin b.$ Then:
\begin{equation}\label{picomcomeq}
\pi=\frac{1}{2\sin a}\frac{\displaystyle{\boldsymbol{K}\left(\frac{\sqrt{\sin^2a-\sin^2b}}{\sin a}\right)}}{\displaystyle{\mathrm{F}_{D}^{(4)}\left( \left. 
\begin{array}{c}
\frac32;\frac12,\frac12,\frac12,\frac12 \\[2mm]
2
\end{array}
\right| 1+e^{2ia},1+e^{-2ia},1+e^{2ib},1+e^{-2ib}\right)}}
\end{equation}
\end{teorema}
\begin{proof}
The Roberts approach, built on the Cauchy-Schl\"omilch  transformation \eqref{cs}, has to be applied splitting the integration domain $[0,\infty)$ in $[0,1]$  and  $[1,\infty)$: going on with the theorem \ref{2due} we find that the integral \eqref{intcomcom} is given by:
\begin{equation}\label{intcomcomb}\tag{\ref{intcomcom}b}
_0R_0^{[0,\infty)}(a,b)=\int_0^\infty\frac{{\rm d}u}{\sqrt{\left(4 \sin ^2a+u^2\right) \left(4 \sin ^2b+u^2\right)}}
\end{equation}
but \eqref{intcomcomb} comes by entry 3.152-1 p. 276 of \cite{gr}, so that
\begin{equation}\label{intcomcomc}\tag{\ref{intcomcom}c}
_0R_0^{[0,\infty)}(a,b)=\frac{1}{2\sin a}\boldsymbol{K}\left(\frac{\sqrt{\sin^2a-\sin^2b}}{\sin a}\right)
\end{equation}
On the other side the same integral is computed also through the lemma \ref{iperlemcomcom} as $L(e^{ia},e^{ib})$ in equation \eqref{allcomp}. Thesis \eqref{picomcomeq} follows as before by comparison.
\end{proof}
\begin{osservazione}
If $a\to b$ then theorem \ref{picomcom} does not provide any more a $\pi$ expression because the elliptic integral degenerates in an elementary one
\[
\lim_{a\to b}\int_0^\infty\frac{{\rm d}u}{\sqrt{\left(4 \sin ^2a+u^2\right) \left(4 \sin ^2b+u^2\right)}}=\int_0^\infty\frac{{\rm d}u}{4 \sin ^2a+u^2}=\frac{\pi}{4\sin a}
\]
whilst the Lauricella function $F_D^{(4)}$ has two couples of equal arguments so that it collapses in a two-variable hypergeometric functions of Appell according to the general relationship:
\[
\mathrm{F}_{D}^{(4)}\left( \left. 
\begin{array}{c}
a;b_1,b_2,b_3,b_4 \\[2mm]
c
\end{array}
\right| x,x,y,y\right)=\mathrm{F}_{1}\left( \left. 
\begin{array}{c}
a;b_1+b_2,b_3+b_4 \\[2mm]
c
\end{array}
\right| x,y\right)
\]
But, in the specific case $a=3/2,\,b_1=b_2=b_3=b_4=1/2,\,c=2$ and therefore, being the sum of both $b$ parameters equating the parameter $c$, we thus have the further reduction:
\[
\mathrm{F}_{1}\left( \left. 
\begin{array}{c}
3/2;1,1 \\[2mm]
2
\end{array}
\right| x,y\right)=\frac{2 \left(\sqrt{1-y}-\sqrt{1-x}\right)}{(x-y)\sqrt{(1-x)(1-y)}
    }
\]
which, written for $x=1+e^{2ia},\,y=1+e^{-2ia}$ and compared with the right hand side leads simply to the Euler identity:
\[
\frac{2}{\sqrt{-e^{-2i a}}+\sqrt{-e^{2i a}}}=\frac{1}{\sin a}
\]
\end{osservazione}
\section{Case $\boldsymbol{n=2}$}
\subsection{The sub-case of 8 real roots}
We will use the reduction formula \eqref{rob2} applied to the three integrals, assuming as before $1<a<b$, we get:
\begin{align}
_8R_2^{[0,1/b]}(a,b):=&\int_0^{\frac{1}{b}}\frac{x^2}{\sqrt{(x^2-a^2)(x^2-\frac{1}{a^2})(x^2-b^2)(x^2-\frac{1}{b^2})}}\,{\rm d}x\label{unodua}\\
_8R_2^{[1/a,a]}(a,b):=&\int_{\frac{1}{a}}^{a}\frac{x^2}{\sqrt{(x^2-a^2)(x^2-\frac{1}{a^2})(x^2-b^2)(x^2-\frac{1}{b^2})}}\,{\rm d}x\label{unodub}\\
_8R_2^{[b,\infty)}(a,b):=&\int_b^{\infty}\frac{x^2}{\sqrt{(x^2-a^2)(x^2-\frac{1}{a^2})(x^2-b^2)(x^2-\frac{1}{b^2})}}\,{\rm d}x\label{unoduc}
\end{align}
and reply the path done in theorems \ref{1uno}, \ref{2due} and \ref{3tre} obtaining further formul\ae  \   which provide $\pi$ in theorems \ref{11uno}, \ref{22due} e \ref{33tre}. 
\subsubsection*{Interval $\boldsymbol{[0,1/b]}$}

Computing $_8R_2^{[0,1/b]}(a,b)$ we get:
\begin{teorema}\label{11uno}

If $1<a<b$ then
\begin{equation}\label{11unoth}
\pi=\frac{2b^3}{\displaystyle{\mathrm{F}_{D}^{(3)}\left( \left. 
\begin{array}{c}
\frac32;\frac12,\frac12,\frac12 \\[2mm]
2
\end{array}
\right|\frac{a^2}{b^2},\frac{1}{a^2b^2},\frac{1}{b^4}\right)}}\left(\frac{1}{b^2-1}\,\boldsymbol{K}\left(\frac{b\left(a^2-1\right)}{a \left(b^2-1\right)}\right)-\frac{1}{b^2+1}\,\boldsymbol{K}\left(\frac{b\left(a^2+1\right)}{a \left(b^2+1\right)}\right)\right)
\end{equation}
\end{teorema}
\begin{proof}
Applying the reduction \eqref{rob2} one finds, (we omit at all details for shortness: see theorem \ref{1uno}),
\begin{equation}\label{11unointer}
_8R_2^{[0,1/b]}(a,b)=\frac12\left(\frac{1}{b^2-1}\,\boldsymbol{K}\left(\frac{b\left(a^2-1\right)}{a \left(b^2-1\right)}\right)-\frac{1}{b^2+1}\,\boldsymbol{K}\left(\frac{b\left(a^2+1\right)}{a \left(b^2+1\right)}\right)\right)
\end{equation}
Thesis \eqref{11unoth} follows by lemma \ref{iperlem} which provides:
\begin{equation}\label{11unointerv}
_8R_2^{[0,1/b]}(a,b)=\frac{\pi}{4b^3}\,\mathrm{F}_{D}^{(3)}\left( \left. 
\begin{array}{c}
\frac32;\frac12,\frac12,\frac12 \\[2mm]
2
\end{array}
\right|\frac{a^2}{b^2},\frac{1}{a^2b^2},\frac{1}{b^4}\right)
\end{equation}
\end{proof}
As for corollary \ref{cor1uno} for $a\to1$ by \eqref{11unoth} we get a formula for $\pi$ with only one parameter:
\begin{corollario}\label{cor11uno}
If $b>1$, then:
\begin{equation}\label{cor11uno1}
\pi=\frac{2 b^4(b^2-1)}{1+b^2}\,\frac{\boldsymbol{K}\left(\frac{2b}{1+b^2}\right)}{b^4+(1-b^2)\,{\rm F}_1\left( \left. 
\begin{array}{c}
\frac32;1,\frac12 \\[2mm]
2
\end{array}
\right|\frac{1}{b^2},\frac{1}{b^4}\right)}
\end{equation}
\end{corollario}
\subsubsection*{Interval $\boldsymbol{[1/a,a]}$}
Computing $_8R_2^{[1/a,a]}(a,b)$ we get:
\begin{teorema}\label{22due}
If $1<a<b$, then
\begin{equation}\label{22dueth}
\pi=\frac{2 \sqrt{\left(b^2-a^2\right) \left(a^2 b^2-1\right)}}{a \left(b^2-1\right)}\frac{\boldsymbol{K}\left(\frac{b\left(a^2-1\right)}{a \left(b^2-1\right)}\right)}{\mathrm{F}_{D}^{(3)}\left( \left. 
\begin{array}{c}
\frac12;-\frac12,\frac12,\frac12 \\[2mm]
1
\end{array}
\right|1-a^4,\frac{\left(a^4-1\right) b^2}{a^2-b^2},\frac{a^4-1}{a^2 b^2-1}\right)}
\end{equation}

\end{teorema}
\begin{proof}
For computing $_8R_2^{[1/a,a]}(a,b)$ we meet the same feature seen for the proof of \ref{2due}, where the employ of Cauchy-Schl\"omilch transformation requires to split the integration domain.We are then led to:
\begin{equation}\label{22dueinter}
_8R_2^{[1/a,a]}(a,b)=\frac{b}{b^2-1}\,\boldsymbol{K}\left(\frac{b\left(a^2-1\right)}{a \left(b^2-1\right)}\right)
\end{equation}
On the other side, by formula \eqref{duet} proved in lemma \ref{iperlem}we have:
\begin{equation}\label{22dueinterb}
_8R_2^{[1/a,a]}(a,b)=\frac{\pi  a b}{2 \sqrt{b^2-a^2} \sqrt{a^2 b^2-1}}\,\mathrm{F}_{D}^{(3)}\left( \left. 
\begin{array}{c}
\frac12;-\frac12,\frac12,\frac12 \\[2mm]
1
\end{array}
\right|1-a^4,\frac{\left(a^4-1\right) b^2}{a^2-b^2},\frac{a^4-1}{a^2 b^2-1}\right)
\end{equation}
Thesis \eqref{22dueth} follows by comparing \eqref{22dueinter} and \eqref{22dueinterb}.
\end{proof}
\begin{osservazione}
Let us note that, even with reference to our previous papers \cite{jnt1,jnt2,jnt3,jmaa1} how in \eqref{22dueth} is the first appearance of a Lauricella function whose parameters $b_i$ are not all equal. Then the variables order cannot be changed.
\end{osservazione}
Small changes to theorem \ref{3tre} proof, allow, starting by integral $_8R_2^{[b,\infty)}(a,b)$, to obtain the third formula to $\pi$ when the polynomial \eqref{nnomio} has 8 real and different roots.
\subsubsection*{Interval $\boldsymbol{[b,\infty)}$}
\begin{teorema}\label{33tre}
If $1<a<b$ then:
\begin{equation}\label{33treth}
\pi=\frac{\sqrt{\left(b^2-a^2\right) \left(a^2 b^2-1\right)}}{a b^2 \sqrt{b^4-1}}\,\frac{\displaystyle{(b^2+1)\boldsymbol{K}\left(\frac{b\left(a^2-1\right)}{a \left(b^2-1\right)}\right)+(b^2-1)\boldsymbol{K}\left(\frac{b\left(a^2+1\right)}{a \left(b^2+1\right)}\right)}}{\displaystyle{\mathrm{F}_{D}^{(3)}\left( \left. 
\begin{array}{c}
\frac12;\frac12,\frac12,\frac12 \\[2mm]
1
\end{array}
\right|\frac{1}{1-b^4},\frac{1}{1-a^2 b^2},\frac{a^2}{a^2-b^2}\right)}}
\end{equation}
\end{teorema}
\begin{proof}
By applying \eqref{cs} to the integral $_8R_2^{[b,\infty)}(a,b)$, taking into account the Cauchy-Schl\"omilch is now increasing, we get:
\begin{equation}\label{33trethinter}
_8R_2^{[b,\infty)}(a,b)=\frac12\left(\frac{b}{b^2-1}\boldsymbol{K}\left(\frac{b\left(a^2-1\right)}{a \left(b^2-1\right)}\right)+\frac{b}{b^2+1}\boldsymbol{K}\left(\frac{b\left(a^2+1\right)}{a \left(b^2+1\right)}\right)\right)
\end{equation}
On the other side, through formula \eqref{tret} (proved in lemma \ref{iperlem}) also we get:
\begin{equation}\label{33treinterb}
_8R_2^{[b,\infty)}(a,b)=\frac{\pi  a b^3}{2 \sqrt{\left(b^4-1\right) \left(b^2-a^2\right) \left(a^2 b^2-1\right)}}\,\mathrm{F}_{D}^{(3)}\left( \left. 
\begin{array}{c}
\frac12;\frac12,\frac12,\frac12 \\[2mm]
1
\end{array}
\right|\frac{1}{1-b^4},\frac{1}{1-a^2 b^2},\frac{a^2}{a^2-b^2}\right)
\end{equation}
As usual, thesis \eqref{33treth} follows by comparing \eqref{33trethinter} and \eqref{33treinterb}.
\end{proof}
In this case again we can take the limit for $a\to1^{+}$ getting by \eqref{33treth} a formula to $\pi$ with only one parameter like the formula \eqref{koro3} shown in \ref{kkoro3}.
\begin{corollario}\label{33trecor}
If $b>1$ we have
\begin{equation}\label{33trecorth}
\pi=\frac{\displaystyle{2(1+b^2)^2}\boldsymbol{K}\left(\frac{2b}{1+b^2}\right)}{\displaystyle{1-b^4+2b^2\sqrt{b^4-1}\,\mathrm{F}_{1}\left( \left. 
\begin{array}{c}
\frac12;1,\frac12 \\[2mm]
1
\end{array}
\right|\frac{1}{1-b^2},\frac{1}{1-b^4}\right)}}
\end{equation}
\end{corollario} 
\begin{proof}
If in \eqref{33treth} we take the limit for $a\to1^{+},$ since one of two elliptic integrals becomes of module zero and then degenerates in a elementary one, see equation \eqref{elementare}, we obtain:
\begin{equation}\label{33trecorthinter}
\pi=\frac{\displaystyle{2(b^2+1)^2}\boldsymbol{K}\left(\frac{2b}{1+b^2}\right)}{\displaystyle{1-b^4+2b^2\sqrt{b^4-1}\,\mathrm{F}_{D}^{(3)}\left( \left. 
\begin{array}{c}
\frac12;\frac12,\frac12,\frac12 \\[2mm]
1
\end{array}
\right|\frac{1}{1-b^2},\frac{1}{1-b^2},\frac{1}{1-b^4}\right)}}
\end{equation}
Thesis \eqref{33trecorth} follows by identity \eqref{riduz} already recalled for in the proof of corollary \ref{koro3}.
\end{proof}
\subsection{Sub-case with 4 real and 2 pairs of conjugate complex roots}

If the $P(x)$'s structure is \eqref{nnomio2compl}, we know that $P(x)\geq0$ for $x\in[0,1/b]\cup[b,\infty)$ and then we will consider the integrals:
\begin{align}
_4R_2^{[0,1/b]}(a,b)&=\int_0^{\frac{1}{b}}\frac{x^2}{\sqrt{\left(x^2-\frac{1}{b^2}\right) \left(x^2-b^2\right) \left(x^2-e^{-2 i a}\right)
   \left(x^2-e^{2 i a}\right)}}\,{\rm d}x\label{unoxxb}\\
_4R_2^{[b,\infty)}(a,b)&=\int_b^\infty\frac{x^2}{\sqrt{\left(x^2-\frac{1}{b^2}\right) \left(x^2-b^2\right) \left(x^2-e^{-2 i a}\right)
   \left(x^2-e^{2 i a}\right)}}\,{\rm d}x\label{duexxb}
\end{align}
Through the usual approach of a two-fold evaluation (lemma \ref{iperlemcom}) we will find two further formul\ae  \   to $\pi$. 
\subsubsection*{Interval $\boldsymbol{[0,1/b]}$}
We will start from \eqref{unoxxb}.
\begin{teorema}
If $b>1$ then
\begin{equation}\label{unobbhyy}
\pi=\frac{2 b^4 \left(\left(1+b^2\right) \boldsymbol{K}\left(\frac{2 b \sin a }{\sqrt{b^4-2 b^2 \cos2a+1}}\right)-\sqrt{b^4-2 b^2 \cos2a+1}\, \boldsymbol{K}\left(\frac{2 b \cos a}{1+b^2}\right)\right)}{\displaystyle{\left(1+b^2\right) \sqrt{b^4-2 b^2 \cos2a+1}\,\mathrm{F}_{D}^{(3)}\left( \left. 
\begin{array}{c}
\frac32;\frac12,\frac12,\frac12 \\[2mm]
2
\end{array}
\right| \frac{1}{b^4},\frac{e^{-2ia}}{b^2},\frac{e^{2ia}}{b^2}\right)}}
\end{equation}
\begin{proof}
By lemma \ref{iperlemcom} we have
\begin{equation}\label{unoxxbb}\tag{\ref{unoxxb}a}
_4R_2^{[0,1/b]}(a,b)=\frac{\pi}{4b^3}\,\mathrm{F}_{D}^{(3)}\left( \left. 
\begin{array}{c}
\frac32;\frac12,\frac12,\frac12 \\[2mm]
2
\end{array}
\right| \frac{1}{b^4},\frac{e^{-2ia}}{b^2},\frac{e^{2ia}}{b^2}\right)
\end{equation}
But the transformation \eqref{cs} drives as usually, $_4R_2^{[0,1/b]}(a,b)$ to two  integrals both linked to the entry 3.152-12 p. 277 of \cite{gr}:
\begin{equation}\label{unoxxbbb}\tag{\ref{unoxxb}b}
_4R_2^{[0,1/b]}(a,b)=\frac12\left(\frac{b}{\sqrt{b^4-2b^2\cos2a+1}}\boldsymbol{K}\left(\frac{2 b \sin a }{\sqrt{b^4-2b^2 \cos2a
   +1}}\right)-\frac{b}{1+b^2}\,\boldsymbol{K}\left(\frac{2 b \cos a}{1+b^2}\right)\right)
\end{equation} 
\end{proof}
Thesis \eqref{unobbhyy} follows  by \eqref{unoxxbbb} and \eqref{unoxxbbb}.
\end{teorema}

\subsubsection*{Interval $\boldsymbol{[b,\infty)}$}
Like \eqref{duexxb}, evaluating $_4R_2^{[b,\infty)}(a,b)$, equation \eqref{duexxb} we get

\begin{teorema}\label{2picompl2uno}
If $a,\,b\in\mathbb{R},\,b>1$ then
\begin{equation}\label{ipcom2duetesix}
\pi=\frac{\sqrt{b^4-1} \sqrt{1-2b^2 \cos2a +b^4}}{{\displaystyle{b^2\,\mathrm{F}_{D}^{(3)}\left( \left. 
\begin{array}{c}
\frac12;\frac12,\frac12,\frac12 \\[2mm]
1
\end{array}
\right|\frac{1}{1-b^4},\frac{1}{1-b^2e^{-2 i a}},\frac{1}{1-b^2e^{2 i a} }\right)}}}\left(\frac{\displaystyle{\boldsymbol{K}\left(\frac{2 b \sin a}{\sqrt{1-2b^2 \cos2a+b^4}}\right)}}{\sqrt{1-2b^2 \cos2a +b^4}}+\frac{\displaystyle{\boldsymbol{K}\left(\frac{2 b \cos a}{1+b^2}\right)}}{1+b^2}\right)
\end{equation}
\end{teorema}
\begin{proof}
Thesis follows because lemma \ref{iperlemcom} establishes
\begin{equation}\label{vduehyp}\tag{\ref{duexxb}b}
_4R_2^{[b,\infty)}(a,b)=\frac{b^3\pi}{2\sqrt{b^4-1}\sqrt{1-2\cos2a\, b^2 +b^4}}\,\mathrm{F}_{D}^{(3)}\left( \left. 
\begin{array}{c}
\frac12;\frac12,\frac12,\frac12 \\[2mm]
1
\end{array}
\right|\frac{1}{1-b^4},\frac{1}{1-b^2e^{-2 i a}},\frac{1}{1-b^2e^{2 i a} }\right)
\end{equation}
But through the reduction scheme induced by transformation \eqref{cs}, it is also true that:
\begin{equation}\label{vdueell}\tag{\ref{duexxb}bc}
_4R_2^{[b,\infty)}(a,b)=\frac12\left(\frac{b}{1+b^2}\boldsymbol{K}\left(\frac{2 b \cos a}{1+b^2}\right)+\frac{b}{\sqrt{1-2\cos2a\, b^2 +b^4}}\,\boldsymbol{K}\left(\frac{2 b \sin a}{\sqrt{1-2\cos2a\, b^2 +b^4}}\right)\right)
\end{equation}
then \eqref{ipcom2duetesix} follows by comparison of \eqref{vduehyp} with \eqref{vdueell}.
\end{proof}
\subsection{Sub-case: 4 pairs of coniugate comple roots}
\subsubsection*{Interval $\boldsymbol{[b,\infty)}$}
In such a case we will get a result similar to \ref{picomcom}. The path is now done more quick being easy to check that:
\begin{equation}\label{uggua}
\int_0^\infty\frac{x^2}{\sqrt{\left(1-2x^2  \cos2a\,+x^4\right) \left(1-2 x^2  \cos2b +x^4\right)}}\,{\rm d}x=\int_0^\infty\frac{{\rm d}x}{\sqrt{\left(1-2 x^2 \cos2a +x^4\right) \left(1-2 x^2  \cos2b +x^4\right)}}
\end{equation}
namely, in our notation:
\begin{equation}\label{ugguasym}\tag{\ref{uggua}b}
_0R_2^{[0,\infty)}(a,b)=\,_0R_0^{[0,\infty)}(a,b).
\end{equation}
The integral $_0R_0^{[0,\infty)}(a,b)$ at the right side of \eqref{uggua} has been computed during the proof of theorem \ref {picomcom}, formula \eqref{intcomcomc}. Going into details, we have:
\begin{teorema}\label{picomcomb}
Let $a,\,b\in\mathbb{R}$ be such that $\sin a>\sin b.$ Then:
\begin{equation}\label{picomcomeqb}
\pi=\frac{1}{2\sin a}\frac{\displaystyle{\boldsymbol{K}\left(\frac{\sqrt{\sin^2a-\sin^2b}}{\sin a}\right)}}{\displaystyle{\mathrm{F}_{D}^{(4)}\left( \left. 
\begin{array}{c}
\frac12;\frac12,\frac12,\frac12,\frac12 \\[2mm]
2
\end{array}
\right| 1+e^{2ia},1+e^{-2ia},1+e^{2ib},1+e^{-2ib}\right)}}
\end{equation}
\end{teorema}
\begin{proof}
By lemma \ref{iperlemcomcom}, taking $s=2$ we obtain
\begin{equation}\label{iperlemcococm2}
_0R_2^{[0,\infty)}(a,b)=\frac{\pi}{4}\,\mathrm{F}_{D}^{(4)}\left( \left. 
\begin{array}{c}
\frac12;\frac12,\frac12,\frac12,\frac12 \\[2mm]
2
\end{array}
\right| 1+e^{2ia},1+e^{-2ia},1+e^{2ib},1+e^{-2ib}\right)
\end{equation}
On the other side by \eqref{uggua} and \eqref{intcomcomc} it follows that
\begin{equation}\label{uggla}
_0R_2^{[0,\infty)}(a,b)=\frac{1}{2\sin a}\,\boldsymbol{K}\left(\frac{\sqrt{\sin^2a-\sin^2b}}{\sin a}\right)
\end{equation}
Then thesis \eqref{picomcomeqb} follows by comparing\eqref{iperlemcococm2} with \eqref{uggla}.
\end{proof}
\begin{osservazione}
Both thesis \eqref{picomcomeq} and \eqref{picomcomeqb} of theorems \ref{picomcom} and \ref{picomcomb} differ only in parameter \lq\lq$a$'' of Lauricella's function. 
In other words theorems \ref{picomcom} and \ref{picomcomb} drive to the conclusion that, due to the peculiar choice of the arguments $\left(1+e^{2ia},1+e^{-2ia},1+e^{2ib},1+e^{-2ib}\right)$ both Lauricella functions will coincide.
\end{osservazione}

\section{Case $\boldsymbol{n=4}$}

In the case with exponent 4 we use the relationship \eqref{rob4} and this will require a greater circumspection due to the fact that the Cauchy-Schl\"omich transformation applied to  Roberts's integrals, leads to a difference of two divergent integrals. In addition, the exponent 4 at numerator does not allow any integration on unbounded intervals.
\subsection{Sub-case: 8 real roots}
Here we will face only the integrals:
\begin{align}
_8R_4^{[0,1/b]}(a,b):=&\int_0^{\frac{1}{b}}\frac{x^4}{\sqrt{(x^2-a^2)(x^2-\frac{1}{a^2})(x^2-b^2)(x^2-\frac{1}{b^2})}}\,{\rm d}x\label{unoqua}\\
_8R_4^{[1/a,a]}(a,b):=&\int_{\frac{1}{a}}^{a}\frac{x^4}{\sqrt{(x^2-a^2)(x^2-\frac{1}{a^2})(x^2-b^2)(x^2-\frac{1}{b^2})}}\,{\rm d}x\label{unoqub}
\end{align}
being $_8R_4^{[b,\infty)}(a,b)=+\infty.$

\subsubsection*{Interval $\boldsymbol{[0,1/b]}$}

\begin{teorema}\label{enuth41}
If $1<a<b$ then:
\begin{equation}\label{enuth41f}
\begin{split}
_8R_4^{[0,1/b]}(a,b)&=\frac{1}{2b} \left[\frac{\left(b^4-b^2+1\right) }{\left(1-b^2\right)}\,\boldsymbol{K}\left(\frac{\left(1-a^2\right) b}{a
   \left(1-b^2\right)}\right)+\frac{\left(b^4+b^2+1\right)
   }{b
   \left(1+b^2\right)}\,\boldsymbol{K}\left(\frac{\left(a^2+1\right)}{a \left(b^2+1\right)}\right)\right.\\
   &\left.-\left(1-b^2\right)\,\boldsymbol{E}\left(\frac{\left(1-a^2\right) b}{a
   \left(1-b^2\right)}\right)-\left(1+b^2\right)\,\boldsymbol{E}\left(\frac{\left(1+a^2\right) b}{a \left(1+b^2\right)}\right)\right]
   \end{split}
\end{equation}
\end{teorema}
\begin{proof}
First we will carry out the integration on $[\varepsilon,1/b]$; after the transformation \eqref{cs}, we will take the limit for $\varepsilon\to0^{+}.$ Then we get:
\begin{equation}\label{th41}
\begin{split}
_8R_4^{[0,1/b]}(a,b)&=\lim_{\varepsilon\to0^{+}}\int_\varepsilon^{\frac{1}{b}}\frac{x^4}{\sqrt{(x^2-a^2)(x^2-\frac{1}{a^2})(x^2-b^2)(x^2-\frac{1}{b^2})}}\,{\rm d}x\\
&=\lim_{\varepsilon\to0^{+}}\left(\frac12\int_{b+\frac{1}{b}}^{\varepsilon+\frac{1}{\varepsilon}}\frac{1-u^2}{\sqrt{\left[\left(a+\frac{1}{a}\right)^2-u^2\right]\left[\left(b+\frac{1}{b}\right)^2-u^2\right]}}{\rm d}u\right.\\
&\left.+\frac12\int_{b-\frac{1}{b}}^{\varepsilon-\frac{1}{\varepsilon}}\frac{1+v^2}{\sqrt{\left[\left(a-\frac{1}{a}\right)^2-v^2\right]\left[\left(b-\frac{1}{b}\right)^2-v^2\right]}}\,{\rm d}v\right)
\end{split}
\end{equation}
We must take the limits ought to both integrals of kind $u>\alpha>\beta$
\[
G(u,\alpha,\beta)=\int_\alpha^u\frac{x^2}{\sqrt{(\alpha^2-x^2)(\beta^2-x^2)}}\,{\rm d}x
\]
appearing in \eqref{th41} are divergent for $x\to+\infty$ and in \eqref{th41} one shall evaluate their difference. Now, by the entry 3.153-9 of \cite{gr} we have:
\[
G(u,\alpha,\beta)=\alpha\left[F\left(\arcsin\frac{u}{\beta},\frac{\beta}{\alpha}\right)-E\left(\arcsin\frac{u}{\beta},\frac{\beta}{\alpha}\right)\right]+u\sqrt{\frac{u^2-\alpha^2}{u^2-\beta^2}}
\]
Recalling the integration formula \eqref{63realetoo}, thesis \eqref{enuth41f} follows.
\end{proof}
Using again lemma \ref{iperlem}, through the usual double evaluation, we find a formula providing $\pi$ where complete second kind elliptic integrals appear too.
\begin{teorema}
If $1<a<b$ then:
\begin{equation}\label{42th}
\pi=\frac{8}{3} b^4 \frac{\mathbb{K}+\mathbb{E}}{\mathbb{F}}
\end{equation}
ove
\begin{equation}\label{42thK}\tag{\ref{42th}K}
\mathbb{K}=\frac{\left(b^4-b^2+1\right)}{b^2-1}\,\boldsymbol{K}\left(\frac{\left(a^2-1\right)
   b}{a\left(b^2-1\right)}\right)-\frac{\left(b^4+b^2+1\right)}{1+b^2}\, \boldsymbol{K}\left(\frac{\left(a^2+1\right) b}{a\left(1+b^2\right)}\right)
\end{equation}
and
\begin{equation}\label{42thE}\tag{\ref{42th}E}
\mathbb{E}=\left(1-b^2\right)
   \boldsymbol{E}\left(\frac{\left(a^2-1\right) b}{a
   \left(b^2-1\right)}\right)+\left(1+b^2\right) \boldsymbol{E}\left(\frac{\left(a^2+1\right)
   b}{a\left(1+b^2\right)}\right)
\end{equation}
finally:
\begin{equation}\label{42thF}\tag{\ref{42th}F}
\mathbb{F}=\mathrm{F}_{D}^{(3)}\left( \left. 
\begin{array}{c}
\frac52;\frac12,\frac12,\frac12 \\[2mm]
3
\end{array}
\right| \frac{a^2}{b^2},\frac{1}{a^2b^2},\frac{1}{b^4}\right)
\end{equation}
\end{teorema}
\begin{proof}
Through the lemma \ref{iperlem} for computing $_8R_4^{[0,1/b]}(a,b)$ we get:
\begin{equation}\label{42thp}
_8R_4^{[0,1/b]}(a,b)=\frac{3\pi}{16b^5}\mathrm{F}_{D}^{(3)}\left( \left. 
\begin{array}{c}
\frac52;\frac12,\frac12,\frac12 \\[2mm]
3
\end{array}
\right| \frac{a^2}{b^2},\frac{1}{a^2b^2},\frac{1}{b^4}\right)
\end{equation}
Thesis \eqref{42th} follows by comparison \eqref{enuth41f} to \eqref{42thp} and getting $\pi$.
\end{proof}
In such a case too, it makes sense to take the limit for $a\to1^{+}$ and obtain a formula holding only one parameter where the Lauricella  function degenerates in an Appell one.

\begin{teorema}\label{n=4i1a=1}
If $b>1$, then
\begin{equation}\label{corth42}
\pi=\frac{8 b^4}{1+b^2}\,\frac{\left(b^6-1\right) \boldsymbol{K}\left(\frac{2
   b}{1+b^2}\right)+\left(1-b^2\right)\left(1+b^2\right)^2 \boldsymbol{E}\left(\frac{2
   b}{1+b^2}\right)}{4b^6+3 \left(1-b^2\right)
   {\rm F}_1\left( \left. 
\begin{array}{c}
\frac52;1,\frac12\\[2mm]
3
\end{array}
\right| \frac{1}{b^2},\frac{1}{b^4}\right)
}
\end{equation}
\end{teorema}

\subsubsection*{Interval $\boldsymbol{[1/a,a]}$}
By the transformation \eqref{cs} we compute the integral \eqref{unoqub} and obtain:
\begin{teorema}\label{quaqua}
If$1<a<b$ then:
\begin{equation}\label{44then}
_8R_4^{[1/a,a]}(a,b)=\frac{1}{b(b^2-1)}\left((1-b^2+b^4)\boldsymbol{K}\left(\frac{b(a^2-1)}{a(b^2-1)}\right)-(b^2-1)^2\boldsymbol{E}\left(\frac{b(a^2-1)}{a(b^2-1)}\right)\right)
\end{equation}
\end{teorema}
\begin{proof}
By \eqref{cs} we get:
\begin{equation}\label{44thena}\tag{\ref{44then}b}
_8R_4^{[1/a,a]}(a,b)=\int_0^{a-\frac{1}{a}}\frac{1+v^2}{\sqrt{\left[\left(a-\frac{1}{a}\right)^2-v^2\right]\left[\left(b-\frac{1}{b}\right)^2-v^2\right]}}\,{\rm d}v
\end{equation}
Thesis \eqref{44then} follows from \eqref{44thena}, with the help of entries 3.153-5 p. 277 and 3.152-7 p. 266 of \cite{gr}.
\end{proof}
By lemma \eqref{iperlem} and theorem \ref{quaqua} we obtain:

\begin{teorema}\label{quacin}
If $1<a<b$ then:
\begin{equation}\label{quacinen}
\pi=\frac{2 a \sqrt{b^2-a^2} \sqrt{a^2 b^2-1}}{b^2 \left(b^2-1\right)}\frac{(1-b^2+b^4)\boldsymbol{K}\left(\frac{b(a^2-1)}{a(b^2-1)}\right)-(b^2-1)^2\boldsymbol{E}\left(\frac{b(a^2-1)}{a(b^2-1)}\right)}{\mathrm{F}_{D}^{(3)}\left( \left. 
\begin{array}{c}
\frac12;-\frac32,\frac12,\frac12 \\[2mm]
1
\end{array}
\right| 1-a^4,\frac{\left(a^4-1\right) b^2}{a^2-b^2},\frac{a^4-1}{a^2 b^2-1}\right)}
\end{equation}
\end{teorema}
\subsection{Sub-case: 4 real and 2 pairs of conjugate complex roots}
\subsubsection*{Interval $\boldsymbol{[0,1/b]}$}
In such a case $P(x)$ has the shape \eqref{nnomio2compl}, then, via the \eqref{cs}, we are going to study the integral:
\begin{equation}
_4R_4^{[0,1/b]}(a,b)=\int_0^{\frac{1}{b}}\frac{x^4}{\sqrt{\left(x^2-\frac{1}{b^2}\right) \left(x^2-b^2\right) \left(1-2 \cos2a\,x^2 +x^4\right)}}\,{\rm d}x\label{quatta}
\end{equation}
By means of our usual technicalities, we get the reduction:
\begin{teorema}\label{th46}
Se $b>1$ e $a\in\mathbb{R}$ si ha:
\begin{equation}\label{th46eq}
\begin{split}
_4R_4^{[0,1/b]}(a,b)&=\frac{1}{2 b}\left(\left(1+b^2\right) \mathbb{E}_1(a,b)-\frac{\left(b^4+b^2+1\right)
   }{1+b^2}\,\mathbb{K}_1(a,b)\right.\\
   &\left.+\frac{\left(b^4-b^2+1\right)}{\sqrt{1-2 b^2 \cos
   2 a+b^4}}\, \mathbb{K}_2(a,b)-\sqrt{1-2 b^2 \cos 2 a+b^4}
   \;\mathbb{E}_2(a,b)\right)
   \end{split}
\end{equation}
where
\begin{align}
\mathbb{K}_1(a,b)&=\boldsymbol{K}\left(\frac{2 b \cos a}{1+b^2}\right),\quad \mathbb{E}_1(a,b)=\boldsymbol{E}\left(\frac{2 b \cos a}{1+b^2}\right)\tag{\ref{th46eq}a}\\
\mathbb{K}_2(a,b)&=\boldsymbol{K}\left(\frac{2 b \sin a}{\sqrt{1-2 b^2 \cos 2 a+b^4}}\right),\quad \mathbb{E}_2(a,b)=\boldsymbol{E}\left(\frac{2 b \sin a}{\sqrt{1-2 b^2 \cos 2 a+b^4}}\right)\tag{\ref{th46eq}b}
\end{align}
\end{teorema}
By thesis \eqref{th46eq} and by lemma \ref{iperlemcom}, form. \eqref{uunobb}, we conclude:

\begin{teorema}
If $b>1$ and $a\in\mathbb{R}$ we get:
\begin{equation}
\begin{split}
\pi&=\frac{8b^4}{3\mathrm{F}_{D}^{(3)}\left( \left. 
\begin{array}{c}
\frac52;\frac12,\frac12,\frac12 \\[2mm]
3
\end{array}
\right| \frac{1}{b^4},\frac{e^{-2ia}}{b^2},\frac{e^{2ia}}{b^2}\right)}\left(\left(1+b^2\right) \mathbb{E}_1(a,b)-\frac{\left(b^4+b^2+1\right)
   }{1+b^2}\,\mathbb{K}_1(a,b) \right.
\\
&\left.+\frac{\left(b^4-b^2+1\right)}{\sqrt{1-2 b^2 \cos
   2 a+b^4}}\,\mathbb{K}_2(a,b)-\sqrt{1-2 b^2 \cos 2 a+b^4}
   \;\mathbb{E}_2(a,b)\right)
   \end{split}
\end{equation}
\end{teorema}
\section{Conclusions}
In this article based upon a Roberts's reduction approach of hyperelliptic integrals to elliptic ones, and on the simultaneous multivariable hypergeometric series evaluation of them,  several identities have been obtained expressing $\pi$ in terms of special values of elliptic, hypergeometric and Gamma functions. 
We acted selecting the values $n=0,\, 2,\, 4$ in equation \eqref{robgen} and to the roots of reciprocal equation $P(x)=0$, where $P(x)$ is provided by \eqref{nnomio}.
Accordingly, we obtained formul\ae\, where $\pi$ can be constructed trough only one parameter as in formul\ae\, \eqref{cor1uno1}, \eqref{koro3}, \eqref{cor11uno1}, or through two parameters like \eqref{11unoth}. Let us note the $\pi$-formula \eqref{picomcomeq} holds two parameters and the imaginary unit too.

In any case, such results will be a useful tool in order to check the routines which one can build for the practical computations of Lauricella's functions we too met frequently in our researches on Mechanics or Elasticity, as in \cite{Laur1, Laur2, Laur3, Laur4}.

These relationship are, as far as we are concerned, all unpublished and undoubtably unknown not only to any human being but also to any computer algebra systems like Mathematica$_\text{\textregistered}$.

For example, by implementing our \eqref{cor1uno1} in Mathematica$_\text{\textregistered}$ we get:
\begin{figure}[H]
\begin{center}
\scalebox{0.75}{\includegraphics{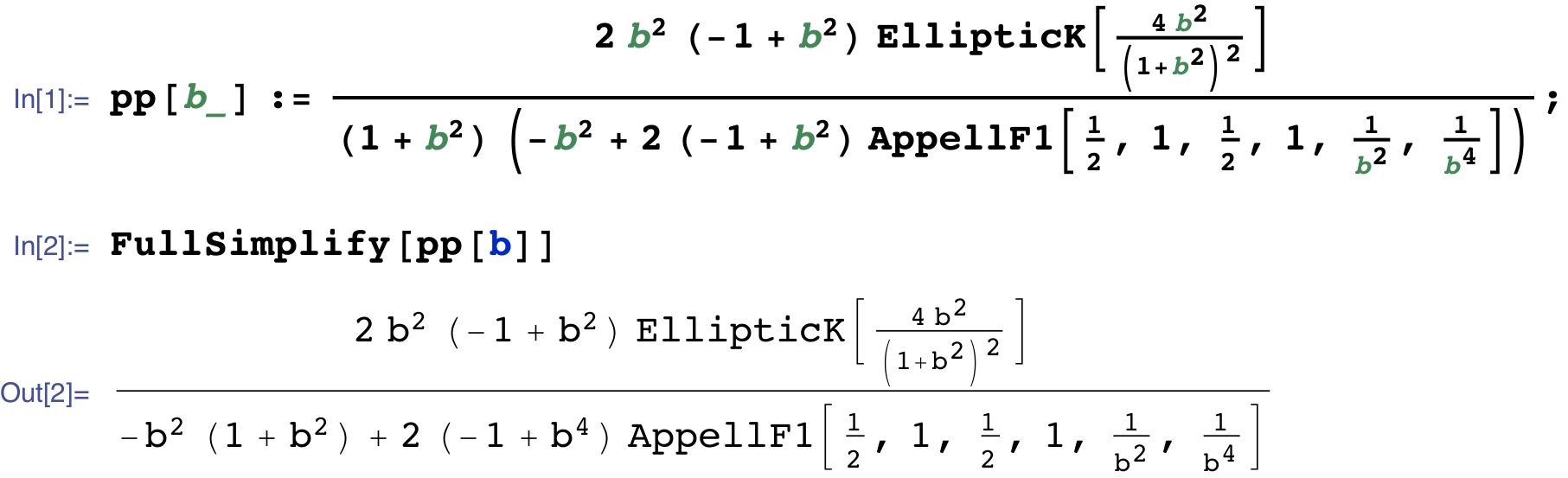}}\label{xf02} 
\end{center}
\caption{Formula \eqref{cor1uno1}}
\end{figure}

\noindent Mathematica$_\text{\textregistered}$ is not able to reduce the above expression to  $\pi$. Nevertheless proceeding numerically we see that the results is confirmed within the range of numerical turbulence.
\begin{figure}[H]
\begin{center}
\scalebox{0.75}{\includegraphics{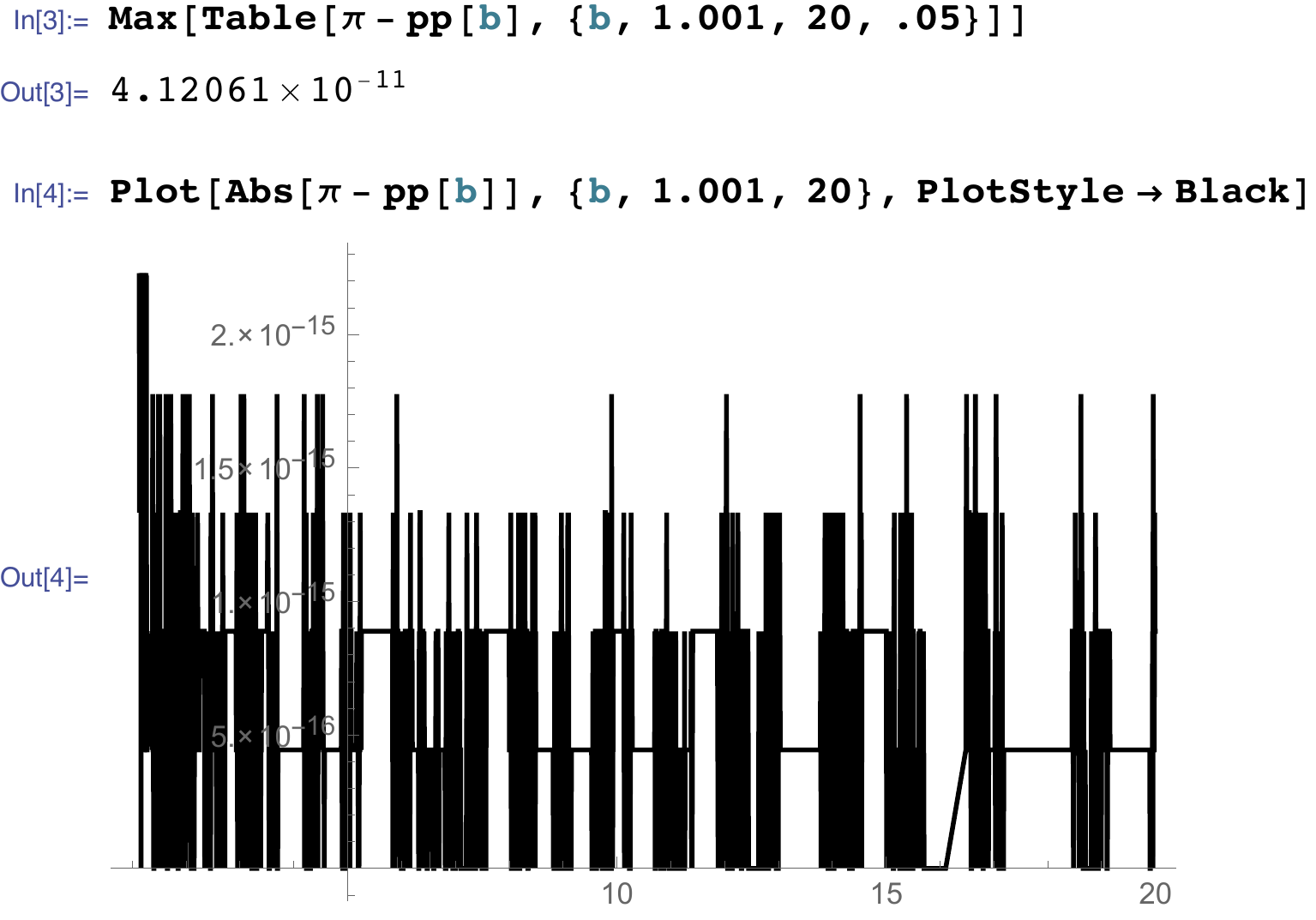}}\label{xf03} 
\end{center}
\caption{Numerical check of formula \eqref{cor1uno1}}
\end{figure}


\begin{thebibliography}{10}
\expandafter\ifx\csname url\endcsname\relax
  \def\url#1{\texttt{#1}}\fi
\expandafter\ifx\csname urlprefix\endcsname\relax\def\urlprefix{URL }\fi
\expandafter\ifx\csname href\endcsname\relax
  \def\href#1#2{#2} \def\path#1{#1}\fi

\bibitem{jnt3}
G.~Mingari~Scarpello, D.~Ritelli, Legendre hyperelliptic integrals, $\pi$ new
  formulae and Lauricella functions through the elliptic singular moduli,
  Journal of Number Theory 135C (2014) 334--352.

\bibitem{jnt1}
G.~Mingari~Scarpello, D.~Ritelli, The hyperelliptic integrals and $\pi$,
  Journal of Number Theory 129 (2009) 3094--3108.

\bibitem{jnt2}
G.~Mingari~Scarpello, D.~Ritelli, $\pi$ and the hypergeometric functions of
  complex argument, Journal of Number Theory 131 (2011) 1887--1900.

\bibitem{jmaa1}
G.~Mingari~Scarpello, D.~Ritelli, On computing some special values of
  multivariate hypergeometric functions, Journal of Mathematical Analysis and
  Applications 420 (2014) 1693--1718.

\bibitem{rob}
M.~Roberts, A Tract on the addition of Elliptic and hyperelliptic integrals,
  Hodger, Foster and Co., 1871.

\bibitem{Laur1}
G.~Mingari~Scarpello, D.~Ritelli, Elliptic integral solutions of spatial
  elastica of a thin straight rod bent under concentrated terminal forces,
  Meccanica 5~(41) (2006) 519--527.

\bibitem{Laur2}
G.~Mingari~Scarpello, D.~Ritelli, Elliptic integrals solution to elastica's
  boundary value problem of a rod bent by axial compression, J. Anal. Appl.
  5~(1) (2005) 53--69.

\bibitem{Laur3}
G.~Mingari~Scarpello, D.~Ritelli, Elliptic functions solution to exact
  curvature elastica of a thin cantilever under terminal loads, J. Geom.
  Symmetry Phys. 12~(1) (2008) 75--92.

\bibitem{Laur4}
G.~Mingari~Scarpello, D.~Ritelli, Exact solutions of nonlinear equation of rod
  deflections involving the Lauricella hypergeometric functions, Int. J. Math.
  Math. Sci. (2011) Art. ID 838924.

\bibitem{roberts}
W.~Roberts, Review of the mathematical papers of the late mr. Michael Roberts,
  Hermathena 5 (1884) 171--185.

\bibitem{Schlo}
T.~Amdeberhan, M.~L. Glasser, M.~C. Jones, V.~Moll, R.~Posey, D.~Varela, The
  Cauchy-Schl\"omilch transformation, arXiv preprint arXiv:1004.2445.

\bibitem{saran1954}
S.~Saran, Hypergeometric functions of three variables, Ganita 5 (1954) 77--91.

\bibitem{lauricella1893}
G.~Lauricella, Sulle funzioni ipergeometriche a pi\`{u} variabili, Rendiconti
  del Circolo Matematico di Palermo 7 (1893) 111--158.

\bibitem{gr}
I.~Gradshtein, I.~Ryzhik, A.~Jeffrey, D.~Zwillinger, Table of integrals,
  series, and products, Academic press, 2007.

\bibitem{by}
P.~Byrd, M.~Friedman, Handbook of elliptic integrals for engineers and
  scientists, Springer Berlin, 1971.

\end{thebibliography}

\end{document}